\documentclass[11pt]{article}
\usepackage{amsfonts}
\usepackage{amsmath}
\usepackage{pdfsync}
\usepackage{enumerate}

\usepackage[all]{xy}
\usepackage[english]{babel}
\usepackage[latin1]{inputenc}
\usepackage[T1]{fontenc}
\usepackage{graphicx,color}
\usepackage{xcolor}
\colorlet{blu1}{blue!70!black}
\colorlet{blu2}{blue!50!black}
\colorlet{blu3}{blue!70!red}
\colorlet{blu4}{blue!60!green}
\colorlet{red1}{red!80}
\colorlet{red2}{red!50!black}
\colorlet{red3}{red!70!yellow}
\colorlet{red4}{red!50!yellow}
\colorlet{yel1}{yellow!50!black}
\colorlet{yel3}{yellow!20!blue}
\colorlet{gre1}{green!60!blue}
\colorlet{gre2}{green!60!black}
\colorlet{gre3}{green!40!black}

\newtheorem{theorem}{Theorem}
\newtheorem{definition}{Definition}
\newtheorem{lemma}{Lemma}
\newtheorem{remark}{Remark}

\newtheorem{proposition}{Proposition}
\newtheorem{example}{Example}

\newtheorem{proof}{Proof}

\numberwithin{equation}{section}

\title{An approach to non-equilibrium Markov chains through cycle matrices}
\author{Cruz de la Rosa M.A. \& Guerrero-Poblete F.\\
\small Universidad Aut\'onoma Metropolitana,  Campus Iztapalapa.  \\
\small CDMX. M\'exico. \\
\small \texttt{marko@xanum.uam.mx}\\
\small \texttt{poblete@xanum.uam.mx}
}

\date{}
\begin{document}

\maketitle

\begin{abstract}
\noindent Analogously to the quantum case considered in \cite{Marco-Fer-I}, this work proposes a graph-theoretic approach to studying non-equilibrium properties in Markov chains. We prove that the kernel of the incidence matrix associated with the interaction graph of the chain, which consists of cycles, is isomorphic to the space of anti-symmetric matrices with rows sum to zero.
The main contribution of this work is the introduction of the called cycle matrices, which constitute a basis for the space of matrices that describe the non-equilibrium. 
\end{abstract}


\section{Introduction}\label{sec:Introd}

\noindent The study of continuous-time Markov chains is a central topic in stochastic process theory. The equilibrium in Markov chains has been widely studied, nevertheless, its counterpart, the non-equilibrium has, until now, not been studied as much. 

It is well known that equilibrium is equivalent to detailed balance condition. The \textit{non-equilibrium} is given by a matrix belonging to the space $\mathcal{M}$ of anti-symmetric matrices with rows that sum to zero.

Analogously to the quantum case considered in \cite{Marco-Fer-I}, this work proposes a graph-theoretic approach to studying non-equilibrium properties in Markov chains. We prove that the kernel of the incidence matrix of the interaction graph and the space $\mathcal{M}$ are isomorphic. In order to obtain a better description of the matrix representing the non-equilibrium, it is convenient to write it in terms of a basis. Since the kernel of an incidence matrix is generated by cycles, we provide a basis for $\mathcal{M}$ in terms  of the images of the cycles under the isomorphism; such matrices will be called \textit{cycle matrices}. 

Once the relationship between the cycles and the non-equilibrium has been established, we focus our attention on certain cycles of intrinsic interest, namely, the Hamiltonian cycles. We give a rule to generate some of them, those that we called \textit{$k$-Hamiltonian cycles}, from these, we introduce the notion of $k$-non-equilibrium for Markov chains and prove that the image of these cycles under the isomorphism are differences between elements in the unitary group of circulant matrices. Finally, we give a complete description of
the $1$-non-equilibrium invariant distribution. 

The paper is organized as follows: in Section 2 we recall the standard notation on Markov chain, we defined the notion of non-equilibrium  and the interaction graph for continuous-time Markov chains. In Section 3 we introduce the notions of k-non-equilibrium and cycle matrices. In Section 4, we give the explicit form of the $1$-non-equilibrium invariant distribution and conclude with an example of a four-state chain.

\section{Non-equilibrium chains and interaction graph}

\noindent Following the standard definitions in Markov chains theory, we shall denote by $q_{ij}$ to the matrix elements of the generator $Q$ and by $\pi_{i}$ to $i$-th element of the invariant distribution $\pi$. Remember that a distribution $\pi$ is invariant if and only if  $\pi Q=0$ and the infinitesimal generator satisfies 
(see \cite{Norris}).
\begin{enumerate}
\item $\displaystyle q_{ij}\geq 0$ for all  $i\neq j$. 
\item $\displaystyle q_{ii}=-\sum_{j: j\neq i}q_{ij}$ for all  $i\in S$. 
\end{enumerate}
\noindent where $S$ is the state space. Next, the definition of detailed balance for a continuous-time Markov chain.
\begin{definition} Given a continuous-time Markov chain $\{X_t\}_{t\geq 0}$ with state space $S$, infinitesimal generator $Q$
and invariant distribution $\pi$. We say that the chain satisfies the \textit{detailed balance condition} if for all $i\neq j$,
\begin{equation}
\label{balance-detallado}\pi_{i} q_{ij}=\pi_{j}q_{ji}
\end{equation}
\end{definition}

\noindent  Equivalently, detailed balance means that for all $i\neq j$,  $\pi_{i} q_{ij}-\pi_{j}q_{ji}=0$. These quantities are known as currents. It is well-established that both conditions, equilibrium and detailed balance are equivalent, see \cite{Stroock} for more details.

%%%%%%%%%%%%%%%%%%%%%%%%%%%%%%%%%%%%%%%%%%%%%%%%%%%%%%%%%%%%%%%%%%%%%%%%%%%%%%%%%

\noindent From now on, we will assume that the state space $S$ has cardinality $N$, and for all $i\neq j$
we have $q_{ij}>0$. Thus, the condition $\pi Q=0$ is equivalent to (see \cite{Marco-Fer})

\begin{eqnarray}\label{Currents-finite}
\sum_{i\neq j}(\pi_{i}q_{ij}-\pi_{j}q_{ji})=0, \qquad \forall j\in S.
\end{eqnarray}

\noindent By denoting $J_{ij}:=(\pi_{i}q_{ij}-\pi_{j}q_{ji})$ and noticing  that $J_{ji}=- J_{ij}$, equation $(\ref{Currents-finite})$ is equivalent to
\begin{eqnarray} \label{Currents-Jij}
\sum_{i=1}^{j-1}J_{ij}-\sum_{i=j+1}^{N}J_{ji}=0\qquad \forall j\in S.
\end{eqnarray} 

\noindent  Let us denote by $\Pi$ the $\footnotesize{N\times N}$ diagonal matrix with entries $\Pi_{i,i}=\pi_{i}$. So, it is clear that detailed balance condition $(\ref{balance-detallado})$ can be written in matrix form as 
\begin{eqnarray}\label{BD-Matriz}
\Pi Q - (\Pi Q)^{t}=0
\end{eqnarray}
\noindent
Since the equilibrium is equivalent to currents equal to zero, the opposite (non equilibrium) means at least one of them is non-zero; formally.

\begin{definition}
Given a continuous-time Markov chain $\{X_t\}_{t\geq 0}$ with finite state space $S$, infinitesimal generator $Q$
and invariant distribution $\pi$,
we say that it is a non-equilibrium Markov chain if and only if
\begin{eqnarray}\label{No-BD-Matriz}
\Pi Q - (\Pi Q)^{t}=D, \qquad D\neq 0.
\end{eqnarray}
\end{definition}

\noindent As one would expect, the matrix $D$ above is not arbitrary and has certain properties, some of which are mentioned below. 

\begin{lemma}\label{D}
The matrix $D$ in $(\ref{No-BD-Matriz})$ fulfills
\begin{enumerate}[i)]
\item $D^{t}=-D$, i.e., it is antisymmetric.
\item $\displaystyle \sum_{j}D_{ij}=0$, \ $\forall i \in S$.
\end{enumerate}
\end{lemma}
\begin{proof}
\noindent Item $i)$  \[D_{ij}^{t}=\pi_{j}q_{ji}-\pi_{i}q_{ij}=-(\pi_{i}q_{ij}-\pi_{j}q_{ji})=-D_{ij}\]
Item $ii)$ follows from item $i)$ and equation $(\ref{Currents-finite})$.
\end{proof}

\subsection{Interaction graph} 

\noindent  The concept of \textit{interaction graph} was introduced first for quantum Markov semigroups of weak coupling limit type in \cite{AccardiFQ}, latter in \cite{Marco-Fer} for continuous-time Markov chains. Namely, given a Markov chain  $\{X_t\}_{t\geq 0}$ with finite state space $S$ and infinitesimal generator $Q$, the interaction graph is defined as the digraph $G(V,E)$, where $V=S$  and $E=\{(i,j)\in S\times S: i<j\}$. 

Notice that, due to our assumption, $q_{ij}>0$ for all $i\neq j$, the interaction graph is complete; hence it has $\binom{N}{2}$ edges. The \textit{incident matrix} of the interaction graph will be denoted by $\Gamma$; in the column indexed by the edge $(i,j)$ such a matrix has $-1$ in the $i$-th row, $1$ in the $j$-th row and $0$ in other case. For any purpose of ordering we will use the lexicographical order,
for such order the function $\theta_N:E\to \{1,\dots,\binom{N}{2}\}$ given by 
\begin{eqnarray}
\theta(i, j)=S(i)+j-i, \quad \text{where}\  S(i):=\sum_{t=1}^{i-1}(N-t)=(i-1)\big(N-\frac{i}{2}\big)\nonumber
\end{eqnarray}
is increasing and bijective  (see \cite{Marco-Fer-Julio} for more details).

We define the \textit{ vector currents} as the vector $J\in \mathbb{R}^{|E|}$, whose entries are precisely the currents  $J_{ij}$ in $(\ref{Currents-Jij})$, namely
$J=\left(J_{12},\cdots, J_{N-1\,N}\right)^{t}$. \\
A crucial fact,  proof of which can be found in \cite{Marco-Fer}, is that $\pi Q=0$ if and only if $\Gamma J=0$. It is well know that cycles in a graph belong to the kernel of its incidence matrix, see \cite{Chartrand}.
Since the minimum length of a cycle is three, for simplicity we seek for cycles of such a length to form a basis for the kernel of the incidence matrix. We will denote by $C_{(i,j,k)}$ to the cycle connecting the vertices $i,j,k$.

Another well known fact is that the rank of an incidence matrix associated with a graph with
$N$ vertices is $N-p$ where $p$ is the number of connected components (see \cite{Busacker}), so, if the graph is complete, i.e., if has $\binom{N}{2} $ edges, its rank is $N-1$ dimensional, hence, its kernel has dimension $\binom{N-1}{2}$.

The following theorem gives a basis for $Ker(\Gamma)$ (see \cite{Marco-Fer-Julio} for a proof).  
\begin{theorem}
Let $\mathcal{C}:=\left\{C_{(i,i+1,i+1+j)}:\, 1\leq i\leq N-2,\, 1\leq j\leq N-(i+1) \right \}$ be a set of minimum length cycles; which, for $r\in \{1,2,\dots, \binom{N}{2}\}$, have coordinates 
\begin{equation}
C_{(i, i+1,i+1+j)}(r)=\left\{
              \begin{array}{rcccl}
                1 & &\text{if} && r=(i-1)(N-\frac{i}{2})+1, \\
                  & &\text{or} && r=i(N-\frac{i+1}{2})+j, \\
                -1 &&\text{if} &&r=(i-1)(N-\frac{i}{2})+j+1, \\
                0 &&  &&\text{other case.}
              \end{array}
            \right.\label{tray cerradas}
\end{equation}
then $\mathcal{C}$ is a basis for $Ker(\Gamma)$.
\end{theorem}

\section{$k$-non-equilibrium}

\noindent As a first step toward the characterization of non-equilibrium distributions in time-continuous Markov chains, since the kernel of the $\Gamma$ matrix consists of cycles, we will focus our attention on certain types of them with intrinsically interesting properties; the Eulerian and Hamiltonian cycles are some of them. Next, in order to handle such cycles, we give a rule to generate them. From now on, we shall use the notation $[i]:=i\, mod\, N$.

\begin{definition}
We shall call \textbf{k-closed-path} to the one that is composed by the sequence $[k],[2k],\dots,[Nk], [k]$. Furthermore, if this path is a cycle (a closed path that does not repeat vertices) it will be called \textbf{$k$-cycle} and denoted by $C_{k}$.  
\end{definition}
Indeed, a $k$-closed-path is not always a $k$-cycle; for example, for $N=6$ and $k=2$ the sequence $[2],[4],[6],[8],[10],[12],[2]$ is a $2$-closed-path but not a $2$-cycle. The following proposition gives a condition for the closed-path to become a cycle, moreover a Hamiltonian cycle, that is, each vertex in the graph is visited exactly once.

\begin{proposition}  \label{com_conn}
The k-closed-path $[k],[2k],\dots,[Nk], [k]$  is a Hamiltonian cycle if and only if $gcd(k,N)=1$.
\end{proposition}
\begin{proof}
It is well known that the cyclic group $\mathbb{Z}/\,N\mathbb{Z}$, i.e., the set $\{1, 2, \cdots,N\}$ joint the operation $i\cdot j:=[i+j]$, has $\psi(N)$ generators, where $\psi$ is the Euler function, and $k\in \mathbb{Z}/\,N\mathbb{Z}$ is a generator for the group if and only if $gcd(k,N)=1$.  
Hence the state space $S=\{1, 2,\dots, N\}=\{[k],[2k],\dots,[Nk]\}$, i.e., each vertex is visited by the sequence $[k],[2k],\dots,[Nk],[k]$ and each one of the edges is different, so the k sequence is a Hamiltonian cycle. 
\end{proof}

\subsection{Cycle Matrices}

\noindent Since the non equilibrium is given by $(\ref{No-BD-Matriz})$ where the matrix $D$ satisfies properties in Lemma $\ref{D}$, in order to obtain a deeper understanding of non-equilibrium  through the structure of $D$, we shall state that $Ker(\Gamma)$ and the space of antisymmetric matrices with rows that sum to zero have the same dimension and therefore are isomorphic. The vector with all its entries equal to one will be denoted by $\mathbf{1}$, i.e., $\mathbf{1}:=(1,\dots, 1)^{t}$.

\begin{theorem}
Let $\mathcal{M}:=\{A: A^{t}=-A,\, A\mathbf{1}=0\}$ be the space of antisymmetric real matrices of size $\footnotesize{N\times N}$ with rows that sum to zero, then $dim(\mathcal{M})=\binom{N-1}{2}$. 
\end{theorem}
\begin{proof}
Let $W=\{A_{N\times N}: A^{t}=-A\}$ be the space of antisymmetric real matrices. It is known that the $dim(W)=\binom{N}{2}=\frac{N(N-1)}{2}$.
Let us consider the linear map $T: W\rightarrow \mathbb{R}^{N}$, given by $T(A)=A\mathbf{1}$.
Notice that $ker T=\{A: A^{t}=-A,\, A\mathbf{1}=0\}=\mathcal{M}$,
so, $\mathcal{M}$ and $ker(T)$ have the same dimension. By Dimension Theorem, 
$dim(ker (T))= dim(W)- dim (Im T)$. We shall prove that $Im T = \mathbf{1}^{\bot}=\{x\in \mathbb{R}^{N}: \langle x, \mathbf{1} \rangle =0 \}$ and since $dim (\mathbf{1}^{\bot})=N-1$, we will conclude that

$$dim (\mathcal{M})=\binom{N}{2}-(N-1)=\binom{N-1}{2}.$$

Let $A\mathbf{1} \in Im T$, then $\langle \mathbf{1}, A\mathbf{1} \rangle= \langle A^{t}\mathbf{1}, \mathbf{1} \rangle = \langle -A\mathbf{1}, \mathbf{1} \rangle= -\langle \mathbf{1}, A\mathbf{1} \rangle.$
 So, $\langle \mathbf{1}, A\mathbf{1} \rangle=0$, that is, $A\mathbf{1}\in \mathbf{1}^{\bot}$, which proves one containment, for the other one, let $x\in \mathbf{1}^{\bot}$, we want to prove that exists $A\in W$, such that, $T(A)=A\mathbf{1}=x$. One possible choice for $A$, is
 
\begin{eqnarray}  A_{i,j}=\left\{\begin{array}{rl}
-x_{j}& \text{for}\ i=1,\ 2\leq j\leq N\\
x_{j}& \text{for}\, j=1,\ 2\leq i\leq N\\
0& \text{other case}
\end{array}\right. 
 \end{eqnarray}
 which clearly satisfies $A=-A^{t}$ and $A\mathbf{1}=x$. This finishes the proof.
\end{proof}

Let us consider the function $\phi:span(\mathcal{C})\to \mathcal{M}$, given by  
\begin{eqnarray}\label{Phi-M}
\phi(C)=M, \qquad \text{where} \quad M_{i,j}=\left\{\begin{array}{rr}
C(\theta(i,j))& \text{if}\, i<j\\
-C(\theta(j,i))& \text{if}\, i>j
\end{array}\right. 
\end{eqnarray}
It is clear that $\phi$ is linear and that $\phi(C)=0$ if and only if $C(\theta(i,j))=0$ for all $i<j$, which means that no one edge belongs to the cycle $C$, i.e., it is the cycle $0$, hence $\phi$ is injective and since $\mathcal{M}$ and $Ker(\Gamma)$ have the same dimension, $\phi$ is onto and therefore an isomorphism. 

\begin{definition}
For any cycle $C$, the matrix $M=\phi(C)$ will be called \textbf{cycle matrix}.
\end{definition}

In particular, for $1\leq i\leq N-2$ and $1\leq j\leq N-(i+1)$, we shall denote as $M_{i,i+1,i+1+j}$ the image under $\phi$ of the cycle $C_{i,i+1,i+1+j}$, such matrices have entries  $(M_{i,i+1,i+1+j})_{mn}$
\begin{equation}
=\left\{
\begin{array}{rcl}
    1 & \text{if}\, & (m,n)\in \{(i,i+1),(i+1,i+1+j),(i+1+j,i)\} \\
    -1 &\text{if}\,&(m,n)\in \{(i+1,i),(i+1+j,i+1),(i,i+1+j)\} \\
    0 &  &\text{other case.}
 \end{array}
 \right.\label{Matriz-ciclos}
\end{equation}
Since the set $\mathcal{C}$ is a basis for $Ker(\Gamma)$, the set of matrices $\{M_{i,i+1,i+1+j}:\,1\leq i\leq N-2,\, 1\leq j\leq N-(i+1)\}$ is a basis for $\mathcal{M}$; hence, equation $(\ref{No-BD-Matriz})$ can be written as 
\begin{eqnarray}\label{No-BD-Matriz-Base}
\Pi Q - (\Pi Q)^{t}=\sum_{i=1}^{N-2} \sum_{j=1}^{N-(i+1)}d_{i,i+1,i+1+j} M_{i,i+1,i+1+j}
\end{eqnarray}
where $d_{i,i+1,i+1+j}\in\mathbb{R}$.

%%%%%%%%%%%%%%%%%%%%%%%%%%%%%%%%%%%%%%%%%%%%%%

\subsection{Circulant matrices}

Once we have established that $Ker(\Gamma)$ and $\mathcal{M}$ are isomorphic, we will look for the counterpart of $C_{k}$ in the space $\mathcal{M}$, that is, $\phi(C_{k})$. There is an interesting relationship between  $\phi(C_{k})$ and the circulant matrices. Recalling that a circulant matrix of order $N$, denoted as $circ(\alpha_{0},\alpha_{1},\dots,\alpha_{N-1})$, is a square matrix such that each row is a cyclic permutation of the previous one, namely
\[
circ(\alpha_{0}, \alpha_{1},...,\alpha_{N-1}):=\left(\begin{array}{llll}
\alpha_{0}&\alpha_{1}&\dots&\alpha_{N-1}\\
\alpha_{N-1}&\alpha_{0}&\dots&\alpha_{N-2}\\
\vdots&\vdots&\ddots&\vdots\\
\alpha_{1}&\alpha_{2}&\dots&\alpha_{0}
\end{array}\right)
\] 
It is known that the set of matrices  $\{\Lambda, \dots ,\Lambda^{N}\}$, where $\Lambda=circ(0,1,0,...,0)$ forms a unitary cyclic group which is also a basis for the space of circulant matrices, i.e., 
\begin{eqnarray} \nonumber 
circ(\alpha_{0}, \alpha_{1},...,\alpha_{N-1})=\alpha_{0}I +\alpha_1 \Lambda+\alpha_2 \Lambda^{2}+\dots +\alpha_{N-1 }\Lambda^{N-1}
\end{eqnarray}
and $(\Lambda^{k})^{-1}=(\Lambda^{k})^{*}=\Lambda^{N-k}$  (see \cite{Ingleton} for a further discussion).

Notice that $(\Lambda^{k}-(\Lambda^{k})^t)\in \mathcal{M}$ for all $1\leq k \leq N$, since it is antisymmetric and $\sum_{j}\big(\Lambda^{k}-(\Lambda^{k})^{t}\big)_{ij}=0$, which is clear from
$$
(\Lambda^{k})_{ij}=\left\{ \begin{array}{cc}1& \text{if\, $ j=[i+k]$}\\0&\text{other case}\end{array}\right.
$$
and
\begin{equation}
\big(\Lambda^{k}-(\Lambda^{k})^{t}\big)_{ij}=\left\{\begin{array}{rr}
1& \text{ if\, $j=[i+k]$}\\
-1& \text{ if\, $i=[j+k]$}
\end{array}\right.\label{lambda_k-lambda_kt}
\end{equation}
On the other hand since $\Lambda^{k}-\Lambda^{N-k}=-(\Lambda^{N-k}- \Lambda^{k})$, it is enough to consider the set $\{ \Lambda^{k}-(\Lambda^{k})^{t}: 1\leq k \leq \ell  \}$ where $\ell$ is such that $N=2\ell$ or $N=2\ell +1$.
 
The following proposition states that the image of the $C_{k}$ under $\phi$ is the differences just described.

\begin{proposition}
Let $C_{k}$ be a $k$-Hamiltonian cycle, then $\phi(C_{k})=\Lambda^{k}-(\Lambda^{k})^{t}$.
\end{proposition}  
\begin{proof}
Since the cycle $\mathcal{C}_{k}=[k],[2k],\dots,[Nk],[k]$ is Hamiltonian; this cycle has the sequence of edges $([k],[2k]),\dots, ([Nk],[(N+1)k])=\{([ik], [(i+1)k]): \, 1\leq i\leq N\}$,
hence, the cycle $C_{k}$  has $1$ in its $r$-th entry if  $[ik]<[(i+1)k]$ and $\theta([ik],[(i+1)k])=r$, i.e., the edge $([ik],[(i+1)k])$ is tra\-ve\-led in the predetermined direction, and has $-1$ if $[(i+1)k]<[ik]$ and $\theta([(i+1)k],[ik])=r$, this is, the edge $([ik],[(i+1)k])$ is traveled in the opposite direction, and $0$ in other case. Now, $[ik]<[(i+1)k]$ if and only if $[(i+1)k]\leq N$, i.e., $i\leq N-k$ and $[(i+1)k]<[ik]$ if and only if $i>N-k$, so  
\[
C_{k}(r)=\left\{ \begin{array}{rl} 1& if\,\,  \theta([ik],[(i+1)k])=r,\, \text{for }\, 1\leq i\leq N-k\\
-1& if\,\,  \theta([(i+1)k],[ik])=r, \, \text{for }\, N-k+1 \leq i\leq N\\
0,&\textrm{other case}\end{array}\right.
\]
Hence  
\[
\phi(C_{k})_{i,j}=\left\{\begin{array}{r}
C_{k}(\theta(i,j)),\, \text{for $i<j$}\, =\left\{\begin{array}{rc}1& \text{if $j=[i+k]$,}\,\,  1\leq i\leq N-k\\-1 & \text{if $i=[j+k],$}\,\, N-k+1\leq j\leq N\end{array}\right.\\
-C_{k}(\theta(j,i)),\, \text{for $i>j$}\,=\left\{\begin{array}{rc}-1& \text{if $i=[j+k]$,}\,\,  1\leq j\leq N-k\\1 & \text{if $j=[i+k],$}\,\, N-k+1\leq i\leq N\end{array}\right.
\end{array}\right. 
\] 
\begin{eqnarray}
=\left\{\begin{array}{rr}
1&\text{if \quad$j=[i+k]$}\\
-1&\text{if \quad$i=[j+k]$}\end{array}\right. \label{Matrix-C-k}
\end{eqnarray}
which is $(\ref{lambda_k-lambda_kt})$.
\end{proof}   
The above result suggests the following definition. 
  
\begin{definition}
Let be a continuous Markov Chain with finite state space $S$, infinitesimal generator $Q$ and invariant distribution $\pi$, we say that $\pi$ is a $k$-non-equilibrium invariant distribution if
\begin{equation}
\Pi Q-(\Pi Q)^{t}=d\big(\Lambda^{k}-(\Lambda^{k})^t\big)\label{k-non_eq}
\end{equation}  
for some $d\in\mathbb{R}$ different from zero.
\end{definition}   
This kind of invariant distributions have been yet considered in the quantum case (see \cite{Marco-Fer-Julio}),
where the  right-side of $(\ref{k-non_eq})$ takes the form $d \sum_{k=1}^{l}\big(\Lambda^{k}-(\Lambda^{k})^t\big)$, these invariant measures   are called
\textit{uniform and completely non equilibrium}.
 
The entries of the matrix equation (\ref{k-non_eq}) are  
\begin{equation}
(\pi_{i}q_{ij}- \pi_{j}q_{ji})_{ij}=\left\{\begin{array}{rl	}
d& \text{ if\, $j=[i+k]$}\\
-d& \text{ if\, $i=[j+k]$}\\
0& \text{other case.}
\end{array}\right.\label{currents-coordinates}
\end{equation}
In the following section we shall describe the Markov chains with an invariant distribution that satisfies (\ref{k-non_eq}) for $k=1$.

\section{$1$-non-equilibrium invariant distribution}

\noindent In this section we shall give an explicitly description of the $1$-non-equilibrium invariant distribution. Since $\Lambda-\Lambda^{t}$ is antisymmetric, it is enough consider only the upper part in (\ref{k-non_eq}), i.e., for $i<j$ we have
 
\begin{equation}\label{Pi-currents-d}
\pi_{i}q_{ij}- \pi_{j}q_{ji}=\left\{\begin{array}{rl}
d& \text{ if\, $j=[i+k]$\ and  $1\leq i \leq N-1$}\\
-d& \text{ if\, $i=[j+k]$\ and $j>N-k$}\\
0& \text{other case}
\end{array}\right.
\end{equation}
So, for $k=1$, the system (\ref{Pi-currents-d}) takes the form

\begin{eqnarray}\label{System-Delta}
 \left(
  \begin{array}{ccccc}
 q_{1,2}&-q_{ 2,1}& 0&\cdots&0\\
  0 &q_{2, 3}&-q_{3,2}&\cdots&0\\
    \vdots &\vdots& \cdots&\ddots&\vdots\\
    %1 &1& 0&0&0 \\
  0&0&\cdots &   q_{N-1, N}&-q_{ N, N-1} \\
    q_{1, N}& \cdots & \cdots &0&-q_{ N, 1} \\
  \end{array}
\right)
\left( \begin{array}{c}
\pi_{1}\\
\pi_{2}\\
\vdots  \\
\pi_{N-1} \\
\pi_{N} \\
\end{array} \right)= \left( \begin{array}{r}
d\\
d\\
\vdots  \\
d \\
-d  \\
\end{array} \right), 
\end{eqnarray}
for short, we simply refer to the above system as $\Delta \pi^{t}=\bar{d}$.

\begin{lemma}
The determinant of the matrix  $\Delta$ in $(\ref{System-Delta})$ is  

\begin{eqnarray}\label{det-Delta}
det(\Delta)=q_{1, N}\prod _{i=1}^{N-1} q_{i+1, i} -q_{ N, 1}\prod _{i=1}^{N-1} q_{i, i+1}
\end{eqnarray}

\end{lemma}

\begin{proof}
The result follows by using the Laplace expansion along the first column and by noticing that in such an expansion we get first an upper triangular matrix followed by a lower triangular matrix, so $$det(\Delta)=- q_{1, 2} \prod _{i=2}^{N-1} q_{i, i+1} q_{ N, 1} + (-1)^{N+1}q_{ N, 1}\prod _{i=1}^{N-1} q_{i+1, i}(-1)^{N-1}.$$ 
Simplifying we get the result.
\end{proof}
\begin{remark} Due to the non-equilibrium hypothesis, $det(\Delta)\neq 0$; since $det(\Delta)= 0$ implies 
$\frac{q_{N, 1}}{q_{1, N}} =\prod _{i=1}^{N-1}\frac{ q_{ i+1, i}}{ q_{i, i+1}}$, i.e., the  Kolmogorov's  reversibility criterion for the cycle $C_{1}$ (see \cite{Durrett}). 
\end{remark}
Next, we consider the system (\ref{Pi-currents-d}) along with the distribution condition $(\sum \pi_{i}=1)$, 
\begin{eqnarray}\label{Augmented-system}
\left(
  \begin{array}{crccc}
  1&1&\cdots &\cdots &1\\
 q_{1, 2}&-q_{ 2, 1}& 0&\cdots&0\\
  0 &q_{2, 3}&-q_{ 3, 2}&\cdots&0\\
    \vdots &\vdots& \cdots&\ddots&\vdots\\
    %1 &1& 0&0&0 \\
  0&0&\cdots &   q_{ N-1, N}&-q_{ N, N-1} \\
    q_{1, N}& \cdots & \cdots &0&-q_{ N, 1} \\
  \end{array}
\right)
\left( \begin{array}{c}
\pi_{1}\\
\pi_{2}\\
\vdots  \\
\pi_{N-1} \\
\pi_{N} \\
\end{array} \right)
=\left(\begin{array}{r}1\\d\\ \vdots \\  d\\-d
\end{array}\right) 
\end{eqnarray}
The {\footnotesize{${(N+1)\times (N+1)}$}} augmented matrix of the above system is
\begin{eqnarray}\label{Augmented-matrix}
\left(
  \begin{array}{crcccr}
  1&1&\cdots &\cdots &1&1\\
 q_{1,2}&-q_{ 2, 1}& 0&\cdots&0&d\\
  0 &q_{2, 3}&-q_{ 3, 2}&\cdots&0&d\\
    \vdots &\vdots& \cdots&\ddots&\vdots& \vdots\\
    %1 &1& 0&0&0 \\
  0&0&\cdots &   q_{ N-1, N}&-q_{ N, N-1}&d \\
    q_{1, N}& \cdots & \cdots &0&-q_{ N, 1} &-d\\
  \end{array}
\right)
\end{eqnarray}
\begin{remark}\label{Determinante-Delta}
Notice that the rank of the above matrix is at least $N$, since the minor $\Delta_{1,N+1}$ is exactly the matrix $\Delta$  in  $(\ref{System-Delta})$.
\end{remark}
\begin{theorem}
Given a continuous-time Markov chain with finite state space $S$, infinitesimal generator $Q$ and $1$-non-equilibrium invariant distribution $\pi$, that is
\begin{eqnarray}\label{1-non-equilibrium}
\Pi Q-(\Pi Q)^{t}=d(\Lambda -\Lambda^{t})
\end{eqnarray}
then, the coordinates of $\pi$ are
\begin{eqnarray}
\pi_{1}=\Big(1+d\sum_{n=2}^{N}\sum_{i=1}^{n-1}\frac{1}{q_{i+1,i}}\prod_{j=i+1}^{n-1}\frac{q_{j,j+1}}{q_{j+1,j}}\Big)\Big(1+ \sum_{n=2}^{N}\prod_{i=1}^{n-1}\frac{q_{i,i+1}}{q_{i+1,i}}\Big)^{-1},\nonumber
\end{eqnarray}
for $2\leq n \leq N$
\begin{eqnarray}\label{pi-non-sim}
\pi_{n}=\pi_{1}\prod_{i=1}^{n-1}\frac{q_{i,i+1}}{q_{i+1,i}}-d\sum_{i=1}^{n-1}\frac{1}{q_{i+1,i}}\prod_{j=i+1}^{n-1}\frac{q_{j,j+1}}{q_{j+1,j}}
\end{eqnarray}
and the unique value $d$ for which $\pi$ is invariant, has the form
\begin{eqnarray}\label{d-value}
d=\frac{(-1)^{N+1}\Delta_{1,N+1}} { \sum_{i=2}^{N} (-1)^{N+1+i} \Delta_{i,N+1}- \Delta_{N+1,N+1}},
\end{eqnarray}
where $\Delta_{j,k}$ is the minor of the augmented matrix in $(\ref{Augmented-matrix})$.

\end{theorem}

\begin{proof}
Since the distribution is 1-non-equilibrium, it satisfies $(\ref{System-Delta})$, i.e., 
\begin{eqnarray}\nonumber
\pi_{n}q_{n,n+1}-\pi_{n+1}q_{n+1,n}&=&d,\quad \text{for}\  n=1,\dots ,N-1 \\ 
\pi_{1}q_{1,N}-\pi_{N}q_{N,1}&=&-d\label{1-jump-db} 
\end{eqnarray}
clearing the first one equation above, we get 
\begin{eqnarray}\nonumber
\pi_{n+1}&=&\pi_{n}\frac{q_{n,n+1}}{q_{n+1,n}}-\frac{d}{q_{n+1,n}}
\end{eqnarray}
recursively, 
\begin{eqnarray}\nonumber
\pi_{n+1}&=&\pi_{n}\frac{q_{n,n+1}}{q_{n+1,n}}-\frac{d}{q_{n+1,n}}=\left( \pi_{n-1}\frac{q_{n-1,n}}{q_{n,n-1}}-\frac{d}{q_{n,n-1}} \right)\frac{q_{n,n+1}}{q_{n+1,n}}-\frac{d}{q_{n+1,n}} \\
&=& \cdots = \pi_{1}\prod_{i=1}^{n-1}\frac{q_{i,i+1}}{q_{i+1,i}}-d\sum_{i=1}^{n-1}\frac{1}{q_{i+1,i}}\prod_{j=i+1}^{n-1}
\frac{q_{j,j+1}}{q_{j+1,j}}
\end{eqnarray}
By using the distribution condition
\begin{eqnarray}
1=\pi_{1}\Big(1+ \sum_{n=2}^{N}\prod_{i=1}^{n-1}\frac{q_{i,i+1}}{q_{i+1,i}}\Big)-d\sum_{n=2}^{N}\sum_{i=1}^{n-1}
\frac{1}{q_{i+1,i}}\prod_{j=i+1}^{n-1}\frac{q_{j,j+1}}{q_{j+1,j}},\nonumber
\end{eqnarray}
clearing the above equation, we get the expression for $\pi_{1}$. Finally to determine the unique value of $d$ such that $\pi$ is an invariant distribution, we proceed as follows.

In order to the vector $(1,d,\dots,d,-d)^{t}$ belongs to the rank of the augmented matrix $(\ref{Augmented-matrix})$, it is necessary that its determinant be equal to zero, so de\-ve\-lo\-ping it along the $N+1$ column, we have
\begin{eqnarray}
0&=&(-1)^{N+2}\Delta_{1,N+1}+d \sum_{i=2}^{N}(-1)^{N+1+i}\Delta_{i,N+1}- d(-1)^{N+1+N+1}\Delta_{N+1,N+1}\nonumber \\
&=&(-1)^{N+2}\Delta_{1,N+1}+d\Big( \sum_{i=2}^{N}(-1)^{N+1+i}\Delta_{i,N+1}-\Delta_{N+1,N+1}\Big)
\end{eqnarray}
i.e., $(-1)^{N+1}\Delta_{1,N+1}=d\Big( \sum_{i=2}^{N}(-1)^{N+1+i}\Delta_{i,N+1}-\Delta_{N+1,N+1}\Big)$. 

By remark $(\ref{Determinante-Delta})$,  $\Delta_{1,N+1}\neq 0$ and by hypothesis $d\neq 0$, then \newline $\Big( \sum_{i=2}^{N}(-1)^{N+1+i}\Delta_{i,N+1}-\Delta_{N+1,N+1}\Big)\neq 0$, hence, by clearing $d$ we have the result. 
\end{proof}
We finish with an example. 

\begin{example}
Let us consider the four states Markov chain, with infinitesimal generator 
\[
Q=\left(\begin{array}{cccc}
q_{11}&q_{12}&q_{13}&q_{14}\\
q_{21}&q_{22}&q_{23}&q_{24}\\
q_{31}&q_{32}&q_{33}&q_{34}\\
q_{41}&q_{42}&q_{43}&q_{44}
\end{array}\right),\qquad q_{ii}=-\sum_{j:j\neq i}q_{ij}.
\]

\begin{figure}[h]
\[\xymatrix{ *+[o][F]{1} \ar[rr] \ar[rdrd] \ar[dd] & &*+[o][F]{2} \ar[ldld] \ar[dd] \\ & & \\
*+[o][F]{4} & & *+[o][F]{3} \ar[ll] }  \]
\caption{Interaction graph.}
\end{figure}
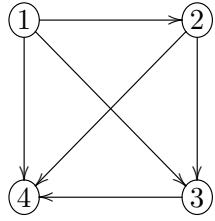 
Its incidence matrix has the form 
\[
\Gamma=\left( \begin{array}{rrrrrr}
-1&-1&-1&0 & 0 &0\\
1 & 0& 0&-1& -1&0\\
0 & 1& 0& 1& 0&-1\\
0 & 0& 1&0& 1&1\\
\end{array}\right).
\]
From $(\ref{tray cerradas})$, a basis for $Ker(\Gamma)$ is the set of cycles $ \{C_{(1,2,3)}, C_{(1,2,4)}, C_{(2,3,4)} \}$, where   
\[
C_{(1,2,3)}=\left(\begin{array}{r}
1\\-1\\0\\1\\0\\0
\end{array}  \right),\qquad 
C_{(1,2,4)}=\left(\begin{array}{r}
1\\0\\-1\\0\\1\\0
\end{array}  \right),\qquad 
C_{(2,3,4)}=\left(\begin{array}{r}
0\\0\\0\\1\\-1\\1
\end{array}  \right)
\]
and, from above discussed, a basis for $\mathcal{M}$ is the set of cycle matrices  

\begin{displaymath}
M_{1,2,3}=\left( \begin{array}{rrrr}
0&1&-1&0\\
-1&0&1&0\\
1& -1& 0&0\\
0 & 0&0&0\\
\end{array}\right),\quad 
M_{1,2,4}=\left( \begin{array}{rrrr}
0&1&0&-1\\
-1&0&0&1\\
0&0& 0&0\\
1&-1&0&0\\
\end{array}\right),
\end{displaymath}
 
\begin{displaymath}
M_{2,3,4}=\left( \begin{array}{rrrr}
0&0&0&0\\
0&0&1&-1\\
0& -1& 0&1\\
0 &1&-1&0\\
\end{array}\right).
\end{displaymath}
If the chain has $1$-non-equilibrium distribution, it satisfies $(\ref{1-non-equilibrium})$, where \newline $\Lambda-\Lambda^{t}=$
\begin{equation}
\left( \begin{array}{rrrr}
0&1&0&0\\
0&0&1&0\\
0& 0& 0&1\\
1 &0&0&0\\
\end{array}\right)-\left( \begin{array}{rrrr}
0&0&0&1\\
1&0&0&0\\
0&1&0&0\\
0&0&1&0\\
\end{array}\right)=\left( \begin{array}{rrrr}
0&1&0&-1\\
-1&0&1&0\\
0&-1&0&1\\
1&0&-1&0\\\end{array}\right).
\end{equation}
Notice that  $\Lambda-\Lambda^{t}= M_{1,2,4}+M_{2,3,4}$.
\end{example}

\bigskip

\noindent \textbf{Conclusions and Perspectives}
\vspace{0.5cm}

\noindent We have introduced the concept of cycle matrices to express the matrix $D$, which describes the non-equilibrium in a Markov chain as a linear combination of them. The $k$-non-equilibrium invariant distribution has been completely described when $D$ takes the form $d\,\phi(C_{1})$, where $C_{1}$ is a Hamiltonian cycle.

The next step is to explicitly describe the $k$-non-equilibrium invariant distribution for $k$ greater than $1$. Additionally, we aim to extend this analysis to describe non-equilibrium invariant distributions for more general and sufficiently interesting cycles, as well as for linear combinations of them.

\end{document}